\documentclass{amsart}
\usepackage{amsmath}
\usepackage{amsfonts}
\usepackage{amssymb}
\usepackage{graphicx}%
\providecommand{\U}[1]{\protect\rule{.1in}{.1in}}
\newenvironment{Question}{\par\noindent\textbf{Question:}\ }{\par}
\usepackage[utf8]{inputenc}
\usepackage{amstext}
\usepackage[all]{xy}
\usepackage[backend=biber]{biblatex}
\usepackage{csquotes}
\bibliography{referencias}

\usepackage{enumerate}
\usepackage{amsthm}
\usepackage{latexsym}

\usepackage{mathrsfs}

\usepackage{amscd}
\usepackage{eucal}
\usepackage{amsxtra}
\usepackage{amsbsy}
\usepackage{graphicx}
\usepackage{latexsym}
\usepackage{pb-diagram}
\usepackage{amsopn}
\usepackage{delarray}
\usepackage{psfrag}
\usepackage{verbatim}
\usepackage{bussproofs}
\usepackage[T1]{fontenc}
\usepackage{color,cancel}
\usepackage[english]{babel}

\usepackage{appendix}
\usepackage{color}
\usepackage[dvipsnames]{xcolor}
\usepackage{epigraph}
\usepackage{lipsum}
\usepackage{tikz}
\usetikzlibrary{positioning}

\usepackage{hyperref} 
\hypersetup{pdfpagemode={UseOutlines},
bookmarksopen=true,bookmarksopenlevel=0,hypertexnames=false,colorlinks=true,citecolor=black,linkcolor=black,urlcolor=black,pdfstartview={FitV},unicode,breaklinks=true}
\usepackage{graphicx}

\newtheorem{mydef}{Definition}%
\newtheorem{theorem}[mydef]{Theorem}
\newtheorem{lemma}[mydef]{Lemma}

\newtheorem{prop}[mydef]{Proposition}
\newtheorem{cor}[mydef]{Corollary}

\newtheorem{claim}[mydef]{Claim}
\newtheorem{claim*}{Claim}

\newtheorem{question}[mydef]{Question}

\title{Shelah ultrafilters}

\author[Balderas]{Emmanuel Balderas}
\address{Posgrado Conjunto en Ciencias Matemáticas UNAM-UMSNH\\Morelia\\Morelia, Michoacán\\México 58089}
\curraddr{}
    \email{ebalderas@matmor.unam.mx}
    \thanks{The first author has been  supported by CONACyT, Scholarship.} 

\author[Chodounský]{David Chodounský}
\address{Institute of Mathematics of the Czech Academy of Sciences, Žitná 25, Praha 1, Czech Republic}
\curraddr{}
    \email{chodounsky@math.cas.cz}
    \thanks{The second author was supported by the Czech Academy of Sciences CAS (RVO 67985840)}
    
\author[Guzmán]{Osvaldo Guzmán}
    \address{Centro de Ciencias Matemáticas\\Universidad Nacional Autónoma de México\\Campus Morelia\\Morelia, Michoacán\\México 58089}
    \curraddr{}
    \email{oguzman@matmor.unam.mx}
    \thanks{The third author was supported by a supported by the PAPIIT grant IA104124 and the CONACyT grant cbf2023-2024-903.}

\keywords{Shelah ultrafilters, Ideals, P-points, Hechler trees, Laver trees}
\subjclass[2010]{03C55, 03C25, 03E05}

\begin{document}

\begin{abstract}
In this paper we study a special type of ultrafilter which we call Shelah ultrafilter. We show that it is possible to add a Shelah ultrafilter using a special forcing notion. We also show that Shelah ultrafilters turn out to be $\mathcal{I}$-ultrafilters for many Borel ideals.
\end{abstract}
\maketitle

\section{Introduction and Preliminaries}

In \cite{Shelahnon-P-points}, Shelah defined a family of non-principal ultrafilters on $\omega$ that seem to be, in some sense, very far from P-points. The primary motivation of Shelah came from forcing and independence results. One reason why P-points are important in set theory is because they behave  nicely respect to certain forcings. In many applications it is important to preserve P-points by a countable support iteration of proper forcings. For P-points these issues are very well understood and many results are known in the literature (see~\cite{Shelahproperforcing}). These results provide a controlled way to get ultrafilters which are generated by $\aleph_1<\mathfrak{c}$ sets. The motivating question of Shelah was whether the methods developed for P-points can also be developed for some other classes of ultrafilters. 
Potentially, the ultrafilters Shelah constructed can be used in the future to prove the consistency of $\mathfrak{u}<\mathfrak{c}$ + \textquotedblleft there are not P-points.\textquotedblright 

In this paper we give another presentation of the special ultrafilters 
constructed by Shelah in~\cite{Shelahnon-P-points}. We also study some of their combinatorial properties which are not necessarily directly related to iterated forcing.\\

A \emph{tree} is a partially ordered set $(T,\leq)$ such that $T$ has a minimum element called the \emph{root} (denoted $rt(T)$) and for every $t\in T$, the set $\{s\in T:s\leq t \}$ is well ordered. As usual, we often write just $T$ instead of $(T,\leq)$.
The members of $T$ are called nodes. For $t\in T$ the immediate successors of $t$ in $T$ is the set $succ_{T}(t)=\{s\in T:t\leq s\wedge\neg(\exists r\in T)(t<r<s)\}$. We define \emph{the cone of $t$ in $T$} as the set $[t]_T=\{s\in T:t\leq s\}$. We say that $S\subseteq T$ is a \emph{subtree} of $T$ if $S\neq\emptyset$ and for every $t\in S$ we have that $\{s\in S:s\leq t\}=\{s\in T:s\leq t\}$. For every $t\in T$ define the set $T_t=\{s\in T: s\leq t \vee t\leq s\}$. Note that $T_t$ is a subtree of $T$ for each $t\in T$. We say that $r\subseteq T$ is a \emph{branch through }$T$ if $r$ is a maximal linearly ordered subset of $T$. 
By $\left[T\right]  $ we denote the set of all branches through $T$. We say that a tree $T$ is well founded if every branch through $T$ is finite. Given a well-founded tree $T$, we say that $s\in T$ is a \emph{leaf} of $T$ if its a maximal node. By $L(T)$ we denote the set of leaves of $T$.\\

An $\textit{ideal}$ on a countable set $X$ is a family of subsets of $X$ that is closed under subsets and finite unions. For every $\mathcal{A}\subseteq\mathcal{P}(X)$ the $\textit{dual family}$ of $\mathcal{A}$ is the set $\mathcal{A^*}=\{X\backslash A:A\in\mathcal{A}\}$. If $\mathcal{I}$ is an ideal on $X$ we say that $A\subseteq X$ is  $\mathcal{I}$ $\textit{positive}$ if $A\notin \mathcal{I}$. The collection of all $\mathcal{I}$ positive sets is denoted by $\mathcal{I}^+$. If $A\in\mathcal{I}^+$, the $\textit{restriction}$ of $\mathcal{I}$ to $A$ is the set $\mathcal{I}\upharpoonright A =\{I\cap A:I\in\mathcal{I}\}$ and it is also an ideal. An ideal $\mathcal{I}$ on $X$ is $\textit{tall}$ if for every $A\in[X]^\omega$ there exists $I\in\mathcal{I}$ such that $|I\cap A|=\omega$.

The dual notion of ideal is called \emph{filter}. A set $\mathcal{F}\subseteq\mathcal{P}(X)$ is a $\textit{filter}$ if $\mathcal{F}^*$ is an ideal. If $\mathcal{F}$ is a filter on $X$ we say that $A\subseteq X$ is $\mathcal{F}$ \emph{positive} if $A$ is $\mathcal{F}^*$ positive; equivalently, $A$ is $\mathcal{F}$ positive if $A\cap F\neq\emptyset$ for every $F\in\mathcal{F}$. In the case that $\mathcal{F}$ is $\subseteq$-maximal, we say that $\mathcal{F}$ is an $\textit{ultrafilter}.$ 

Given a countable set $X$, we can identify $\mathcal{P}(X)$ with the Polish space $2^X$ by associating each $A\subseteq X$ with its characteristic function. In this way, we say that a filter (ideal) on $X$ is Borel (analytic, co-analytic, etc.) if it is Borel (analytic, co-analytic, etc.)  as a subspace of $\mathcal{P}(X)$. 

Let $X$ be a set and $\mathcal{I}$ an ideal on $X.$ If $\varphi\left(
x\right)  $ is a formula, by $\forall^{\mathcal{I}^{\ast}}x\left(
\varphi\left(  x\right)  \right)  $ we mean that the set $\left\{  x\in
X:\varphi\left(  x\right)  \right\}  $ is in $\mathcal{I}^{\ast}.$ Given $X,Y$ two countable sets, $\mathcal{I}$ an ideal on $X$ and
$\mathcal{J}$ an ideal on $Y.$ \emph{The Fubini product of }$\mathcal{I}$
\emph{and} $\mathcal{J}$ is the ideal $\mathcal{I\times J}$ on $X\times Y$ defined as
follows: 
$$A\in\mathcal{I\times J} \textit{    if and only if    } \forall^{\mathcal{I}^{\ast}}b\left(
A(b)\in\mathcal{J}\right).$$

Where $A(b)=\left\{  y\in Y:\left(  b,y\right)  \in A\right\}  $ for every
$b\in X.$

Let $\mathcal{I}$ be an ideal on $X$ and $\mathcal{J}_i$ be ideals on $Y_i$ ($i\in X$). Define $\lim_{i\rightarrow\mathcal{I}}\mathcal{J}_i$ an ideal on $\bigcup_{i\in X} (\{i\}\times Y_i)$ as follows: $$A\in\lim_{i\rightarrow\mathcal{I}}\mathcal{J}_i\textit{ if and only if } \forall^{\mathcal{I}^*}i (A(i)\in\mathcal{J}_i).$$
Observe that in the case that $\mathcal{J}_i=\mathcal{J}$ for every $i\in X$ we have that $\lim_{i\rightarrow\mathcal{I}}\mathcal{J}_i =\mathcal{I}\times\mathcal{J}$.

One of the most important tools for classifying relations between ideals and filters is the $\textit{Kat\v{e}tov order}$ introduced in \cite{RefKatětov}.

\begin{mydef}
Let $\mathcal{I}$ be an ideal on a countable set $X$ and $\mathcal{J}$ be an ideal on a countable set $Y$. We say that $\mathcal{I}$ is Kat\v{e}tov below $\mathcal{J}$ (denoted by $\mathcal{I}\leq_K\mathcal{J}$) if there exists $f:Y\longrightarrow X$ such that $f^{-1}[I]\in\mathcal{J}$ for every $I\in\mathcal{I}$. Such $f$ is called Kat\v{e}tov morphism. If $\mathcal{I}\leq_K\mathcal{J}$ and $\mathcal{J}\leq_K\mathcal{I}$, then we say that $\mathcal{I}$ and $\mathcal{J}$ are Kat\v{e}tov equivalent (denoted by $\mathcal{I}\cong_K\mathcal{J}$).
\end{mydef}

Basic properties of the Kat\v{e}tov are listed on the following Lemma.

\begin{lemma}\label{Basic Katetov}
Let $X$ be a countable set and $\mathcal{I}$ an ideal on $X.$
\begin{enumerate}
\item $fin=\left[  \omega\right]  ^{<\omega}\leq_K\mathcal{I}$.

\item $\mathcal{I}$ and $fin$ are Kat\v{e}tov
equivalent if and only if $\mathcal{I}$ is not tall.

\item If $\mathcal{J}$ is an ideal on $X$ and $\mathcal{J}\subseteq \mathcal{I}$, then $\mathcal{J}\leq_K\mathcal{I}$.

\item If $A\in\mathcal{I}^{+},$ then $\mathcal{I}\leq_K\mathcal{I}\upharpoonright A.$

\item If $A\in\mathcal{I}^{\ast},$ then $\mathcal{I}$ and $\mathcal{I}%
\upharpoonright A$ are Kat\v{e}tov equivalent.
\end{enumerate}
\end{lemma}

To learn more about filters, ideals and the Kat\v{e}tov order the reader can see \cite{COMBFILTERSANDIDEALS}, \cite{KATETOVORDER}, \cite{OrderingMadFamiliesaLaKatětov} and \cite{PIDEALSMICHAELYFER}.

If $\mathcal{I}$ is an ideal on $X$ and $\mathcal{J}$ is an ideal on $Y$, we say that $\mathcal{I}$ is \emph{isomorphic} to $\mathcal{J}$ if there exists a bijective funcion $f: X\longrightarrow Y$ such that $f[A]\in\mathcal{J}$ if and only if $A\in\mathcal{I}$.

\section{Filters generated by trees}

For us, a tree $T$ is $\omega$\emph{-branching }if every node in $T$ is
either a leaf or has $\omega$ many immediate succesors. From now on, we will denote by $\mathbb{W}$ the class of all well-founded, $\omega
$-branching trees. It is easy to prove that every tree in $\mathbb{W}$ is
either finite or countably infinite.

\begin{mydef}
Let $T\in\mathbb{W}$ and $S$ a subtree of $T$.
\begin{enumerate}[\hspace{0.5 cm} (1)]
    \item We say that $S$ is a
Laver subtree of  $T$ if $S\in\mathbb{W}$ and $L(S)\subseteq L(T)$.
\item We say that $S\subseteq T$ is a Hechler subtree of $T$ if $S$ is a Laver subtree of $T$ and for every $t\in S$, we have that   $succ_{T}\left(  t\right)  \backslash succ_{S}\left(  t\right)  $
is finite. 
\end{enumerate}
\end{mydef}

The above definition is based on the definition of Hechler an Laver trees in \cite{miller2012hechlerlavertrees}; however, note that our trees are well founded. We will use $\mathbb{L}(T)$ and $\mathbb{H}(T)$ to denote the set of all Laver and Hechler subtrees of $T$ respectively. Obviously, our notation is inspired by the corresponding forcing notions \cite{Laver}, \cite{Hechler}. The following remarks are immediate from the definitions:

\begin{lemma}\label{subarboles}
Let $T\in\mathbb{W}$ and $S,R$ subtrees of $T.$
\label{Lemma basic Laver Hechler}

\begin{enumerate}
\item $T\in\mathbb{H}\left(  T\right)  .$

\item If $S,R\in\mathbb{H}\left(  T\right)  ,$ then $S\cap R\in
\mathbb{H}\left(  T\right)  .$

\item If $S\in\mathbb{H}\left(  T\right)  $ and $R\in\mathbb{L}\left(
T\right)  ,$ then $S\cap R\in\mathbb{L}\left(  T\right)  .$
\end{enumerate}
\end{lemma}

The previous Lemma tells us that intuitively, we can think that the set of Hechler subtrees of $T$ is the set of \textquotedblleft big\textquotedblright\ subtrees of $T$.

\begin{mydef}
Let $T\in\mathbb{W}$ and $A\subseteq T.$ We say that $A$ is a barrier in
 $T$ if the following holds:

\begin{enumerate}
\item $A$ is an antichain.

\item Every branch through $T$ has an element of $A.$
\end{enumerate}
\end{mydef}

Let $T\in\mathbb{W}$ and $A\subseteq T$ an antichain. We define the infinite
game $\mathcal{G}_{T}\left(  A\right)  $ between Player $\mathsf{I}$ and
Player $\mathsf{II}$ as follows:\\

\hfill%
\begin{center}
\begin{tabular}
[c]{|l|l|l|l|l|l|}\hline
$\mathsf{I}$ & $B_{0}$ &  & $B_{1}$ &  & $...$\\\hline
$\mathsf{II}$ &  & $s_{0}$ &  & $s_{1}$ & $...$\\\hline
\end{tabular}
\end{center}
\hfill
\\

With the following conditions:

\begin{enumerate}
\item $B_{0}$ is a finite subset of $succ_{T}\left(  rt\left(  T\right)
\right)  $ and $s_{0}\in succ_{T}\left(  rt\left(  T\right)  \right)  $ is such
that $s_{0}\notin B_{0}.$

\item If $s_{n}$ is not a leaf, then $B_{n+1}$ is a finite subset of
$succ_{T}\left(  s_{n}\right)  $ and $s_{n+1}\in succ_{T}\left(  s_{n}\right)  $
is such that $s_{n+1}\notin B_{n+1}.$

\item If $s_{n}$ is a leaf, then $B_{n+1}=\emptyset$ and $s_{n+1}=s_{n}.$
\end{enumerate}

\qquad\ \ \ \ \ \ \ 

Player $\mathsf{II}$ wins the match if the sequence $\left\langle
s_{n}\right\rangle _{n\in\omega}$ contains an element of $A$ (note that the
sequence is eventually constant since $T$ is well-founded). If
this is not the case, then Player $\mathsf{II}$ loses. Intuitively, we can imagine the game as follows: Player
$\mathsf{II}$ is standing on the root of $T$ and wants to walk in $T$ until
she reaches a branch. Moreover, for some mysterious reason, she wishes to
pass through an element of $A.$ At each step, Player $\mathsf{I}$ is blocking
finitely many paths to Player $\mathsf{II}$. Player $\mathsf{II}$ will win if she is able to walk through an
element of $A$ in spite of the obstructions placed by Player $\mathsf{I.}%
$\qquad\qquad

\begin{prop}
Let $T\in\mathbb{W}$ and $A\subseteq T$ an antichain. The following holds:
\label{estrategia del juego}

\begin{enumerate}
\item Player $\mathsf{II}$ has a winning strategy in $\mathcal{G}_{T}\left(
A\right)  $ if and only if there is $S\in\mathbb{L}\left(  T\right)  $
such that $S\cap A$ is a barrier in $S.$

\item Player $\mathsf{I}$ has a winning strategy in $\mathcal{G}_{T}\left(
A\right)  $ if and only if there is $S\in\mathbb{H}\left(  T\right)  $
such that $S\cap A=\emptyset.$
\end{enumerate}
\end{prop}

\begin{proof}
We start by proving the first point. If there is $S\in\mathbb{L}\left(
T\right)  $ such that $S\cap A$ is a barrier in $S,$ then Player $\mathsf{II}$
can win the game by making sure she always walks in $S$ (it is possible to do
so, since the moves of Player $\mathsf{I}$ are finite, while non-maximal
points of Laver trees are infinitely branching). Assume now that Player $\mathsf{II}$ has a
winning strategy in $\mathcal{G}_{T}\left(  A\right)  .$ Fix one of her
winning strategies\textsf{.} Let $S\subseteq T$ be the set of all nodes in $T$
that she can reach during a run of the game (in which she is using his fixed
strategy). Since this is a winning strategy, every branch through $S$ must
pass through an element of $A.$ Furthermore, $S$ must be in $\mathbb{L}\left(  T\right)  ,$ since if a non-maximal branch of $S$ was finitely
branching, Player $\mathsf{I}$ could win by playing that finite set.

\qquad\qquad\qquad\qquad\qquad

We now prove the second point. If there is $S\in\mathbb{H}\left(
T\right)  $ such that $S\cap A=\emptyset,$ then Player $\mathsf{I}$ can win by
forcing Player $\mathsf{II}$ to always walk in $S.$ Assume now that she Player $\mathsf{I}$ has a
winning strategy in $\mathcal{G}_{T}\left(  A\right)  .$ Fix one of his
winning strategies$\mathsf{.}$ Let $S\subseteq T$ be the set of all nodes in
$T$ that Player $\mathsf{II}$ can reach during a run of the game (in which
Player $\mathsf{I}$ is using his fixed strategy). Since this is a winning
strategy, no branch of $S$ must pass through an element of $A.$ Furthermore,
$S\in\mathbb{H}\left(  T\right)  $ since Player $\mathsf{I}$ only blocks
finitely many nodes at each move.
\end{proof}

It is clear that $\mathcal{G}_{T}\left(  A\right) $ is a clopen game, i.e., every match is decided in a finite numbers of steps. So, by the Gale-Stewart Theorem, 
we have that  $\mathcal{G}_{T}\left(  A\right) $ is determined (see \cite{Kechris}). Hence we conclude the following dichotomy:

\begin{cor}
Let $T\in\mathbb{W}$ and $A\subseteq T$ an antichain. One and only one of the
following statements holds: \label{corol dicotomia}

\begin{enumerate}
\item There is $S\in\mathbb{H}\left(  T\right)  $ such that $S\cap
A=\emptyset.$

\item There is $S\in\mathbb{L}\left(  T\right)  $ such that $S\cap A$ is a
barrier in $S.$
\end{enumerate}
\end{cor}

There is also a \textquotedblleft dual\textquotedblright\ of the previous
dichotomy, in which the outcome of the two possiblities are switched:

\begin{cor}
Let $T\in\mathbb{W}$ and $A\subseteq T$ an antichain. One and only one of the
following statements holds: \label{corol dicotomia 2}

\begin{enumerate}
\item There is $S\in\mathbb{H}\left(  T\right)  $ such that $S\cap A$ is a
barrier in $S.$

\item There is $S\in\mathbb{L}\left(  T\right)  $ such that $S\cap
A=\emptyset.$
\end{enumerate}
\end{cor}

\begin{proof}
It is clear that at most one of the possibilities may occur. We will now prove
that at least one of them does. Define $B$ as  the collection of leaves that
do not extend an element of $A$. Since $B$ is a set of leaves, it follows that
it is an antichain. We can now apply the Corollary \ref{corol dicotomia} to
$B.$

First assume that there is $S\in\mathbb{H}\left(  T\right)  $ such that
$S\cap B=\emptyset.$ In this way, if $s\in L\left(  S\right)  ,$ then $s\notin
B,$ so $s$ extends an element of $A$. It follows that $A$ is a barrier in
$S.$

Now assume that there is $S\in\mathbb{L}\left(  T\right)  $ such that
$S\cap B$ is a barrier in $S.$ In this way, $L\left(  S\right)  \subseteq B,$
so it must be the case that $S\cap A=\emptyset.$
\end{proof}

\begin{mydef}
Let $T$ be an element of $\mathbb{W}$. Define $\mathcal{F}\left(  T\right)  $ as the filter in
$L\left(  T\right)  $ generated by $\{L(S):S\in\mathbb{H}(T)\}.$
\end{mydef}

Note that $\mathcal{F}\left(  T\right)  $ is really a filter by Lemma
\ref{Lemma basic Laver Hechler}. Furthermore, if $T\in\mathbb{W}$ it is easy to see that every cofinite subset
of $L\left(  T\right)  $ belongs to $\mathcal{F}\left(  T\right)  .$ Ultrafilters defined with trees have been studied in \cite{CanonicalCofinalMaps}, and \cite{CofinalTypesofUltrafilters}. 

Let $T\in\mathbb{W}$ and
$\left\{  s_{n}: n\in\omega\right\}$ be an enumeration of $succ_T(rt(T))$. For every $n\in\omega,$ 
define $S^{n}=T_{s_{n}}\setminus\left\{  rt\left(  T\right)  \right\} $
view it as a suborder of $T.$ Note that for every $X\subseteq L\left(
T\right)  ,$ the following holds:

\begin{center}%
\begin{tabular}
[c]{l}%
$X\in\mathcal{F}\left(  T\right)  $ if and only if $\forall^{fin^*}n\left(
X\in\mathcal{F}(S^{n})\right).  $%
\end{tabular}
\end{center}

We have the following characterization of the $\mathcal{F}\left(  T\right)  $
positive sets:

\begin{prop}
Let $T\in\mathbb{W}$ and $X\subseteq L\left(  T\right)  .$
\label{dicotomia conjunto}

\begin{enumerate}
\item One and only one of the following conditions holds:

\begin{enumerate}
\item There is $S\in\mathbb{L}\left(  T\right)  $ such that $L\left(
S\right)  \subseteq X.$

\item There is $S\in\mathbb{H}\left(  T\right)  $ such that $L\left(
S\right)  \cap X=\emptyset.$
\end{enumerate}

\item $X\in\mathcal{F}\left(  T\right)  ^{+}$ if and only if there is
$S\in\mathbb{L}\left(  T\right)  $ such that $L\left(  S\right)  \subseteq
X.$
\end{enumerate}
\end{prop}

\begin{proof}
We start by proving the first point. Note that $X$ itself is an antichain of
$T.$ A straightforward application of Corollary \ref{corol dicotomia} gives
the desired result. Finally, note that the second point is a consequence of
the first.
\end{proof}

Let $\alpha$ be a limit countable ordinal. By $bnd(\alpha)$ we denote the ideal of all bounded subsets of $\alpha$ (in
other words, $bnd(\alpha)$ is generated by $\left\{
\beta:\beta<\alpha\right\}  $). It is easy to see that this is an ideal.
The following ideals were introduced by Kat\v{e}tov in
\cite{OnDescriptiveClassificationofFunctions}:

\begin{mydef}
For every countable ordinal $\alpha,$ we will define a countable set
$X_{\alpha}$ and an ideal $fin^\alpha$ on $X_{\alpha}$ as follows:
\label{fin alpha}

\begin{enumerate}
\item $X_{0}=\left\{  0\right\}  $ and $fin^0=\left\{
\emptyset\right\}  .$

\item $X_{1}=\omega$ and $fin^1=fin$

\item $X_{\alpha+1}=\omega\times X_{\alpha}$ and $fin^{\alpha+1}=fin\times fin^\alpha.$

\item Let $\alpha$ be a limit ordinal. Define $X_{\alpha}=\bigcup_{\beta<\alpha}\left(  \left\{  \beta\right\}  \times X_{\beta}\right)$ and $fin^\alpha=\lim_{\beta\rightarrow bnd(\alpha)} fin^\beta$.
\end{enumerate}
\end{mydef}

In \cite{KatetovOrderImply} it was proved the following important result.

\begin{theorem}[Barbarski, Filipów, Mrozek and Szuca]
Let $\mathcal{I}$ be an ideal on a countable set $X$ and $\alpha<\omega_1$. If $fin^\alpha\leq_K\mathcal{I}$, then there exists a bijection then there is a bijection witnessing it.
\end{theorem}

In \cite{StructuralKatetov} it was proved that if $\alpha<\beta,$ then
$fin^{\beta}$ is strictly Kat\v{e}tov above $fin^{\alpha}.$ So, using the above Theorem we can conclude the following.

\begin{cor}\label{copiaisomorfa}
Let $\mathcal{I}$ be an ideal Kat\v{e}tov equivalent to $fin^\alpha$ for some $\alpha\in\omega_1$. Then $\mathcal{I}$ contains an isomorphic copy of $fin^\beta$ for every $\beta\leq \alpha$, i.e, there exists $\mathcal{J}$ an ideal included in $\mathcal{I}$ such that $\mathcal{J}$ is isomorphic to $fin^\beta$.
\end{cor}
The following is well-known, we prove it here for completeness:

\begin{lemma}\label{lemma uniforme}
Let $\alpha<\omega_{1}$ be a limit ordinal and $B$ be an unbounded subset of $\alpha$.
Define $\overline{B}=\bigcup_{\beta\in B}(\{\beta\}\times X_\beta)$. Then the ideals
$fin ^{\alpha}$ and $fin ^{\alpha}\upharpoonright\overline{B}$
are Kat\v{e}tov equivalent.  
\end{lemma}
\begin{proof}
It is clear that $\overline{B}$ is not in $fin^{\alpha},$ so by Lemma
\ref{Basic Katetov}, we know that $fin^{\alpha}\upharpoonright
\overline{B}$ is Kat\v{e}tov above $fin ^{\alpha}.$ It remains to
prove that $fin^{\alpha}\upharpoonright\overline{B}$ $\ \leq
_{\text{\textsf{K}}} fin^{\alpha}.$ Once again by Lemma \ref{Basic Katetov}, we may assume that the order type of $B$ is $\omega$ and
that $0\in B.$

Given $\beta<\alpha,$ define $\beta^{-}$ as the largest element of $B$ such that
$\beta^{-}\leq\beta.$ Fix $g_{\beta}:X_{\beta}\longrightarrow X_{\beta^{-}}$ a
Kat\v{e}tov function from $(X_{\beta}$,$fin^{\beta})$ to
$(X_{\beta^{-}}$,$fin^{\beta^{-}}).$ We now define $h:X_{\alpha
}\longrightarrow\overline{B}$ where $h=\bigcup_{\beta<\alpha}
g_{\beta}.$ We claim that $h$ is a Kat\v{e}tov function from $(X_{\alpha}
,fin^{\alpha})$ to $(\overline{B},fin^{\alpha
}\upharpoonright\overline{B}).$

Let $A\subseteq\overline{B}$ with $A\in$ $fin^{\alpha}.$ We need to
prove that $h^{-1}\left(  A\right)  \in$ $fin^{\alpha}.$ We know that
there is $\gamma\in B$ such that if $\beta\geq\gamma,$ then $A_{\beta}\in$
$fin^{\alpha}$. Let $\beta<\alpha$ such that $\gamma<\beta$ (so
$\gamma\leq\beta^{-}$). We have the following:
$$h^{-1}\left(  A\right)  _{\beta}=g_{\beta}^{-1}\left(  A_{\beta^{-}}\right).
$$
Since $g_{\beta}$ is a Kat\v{e}tov function and $A_{\beta^{-}}\in fin^{\beta^{-}}$ it follows that $h^{-1}\left(  A\right)  _{\beta
}=g_{\beta}^{-1}\left(  A_{\beta^{-}}\right)  \in fin^{\beta}.$
From this we conclude that $h^{-1}\left(  A\right)  \in$ 
$fin^{\alpha}.$
\end{proof}

The following is an adaptation of the well-known notion of uniform barrier to
our context (see \cite{Ramseyspaces}):

\begin{mydef}
Let $T\in\mathbb{W}$, $s\in T$ and $\alpha<\omega_{1}.$

\begin{enumerate}
\item We say that $s$ is $0$\emph{-uniform in }$T$ if $s$ is a leaf.

\item We say that $s$ is $\left(  \alpha+1\right)  $\emph{-uniform in }$T$ if
every $t\in succ_{T}\left(  s\right)  $ is $\alpha$-uniform in $T.$

\item Let $\alpha$ be a limit ordinal. We say that $s$ is $\alpha
$\emph{-uniform in }$T$ if there is an enumeration $succ_{T}\left(  s\right)
=\left\{  t_{n}: n\in\omega\right\}  $ and an increasing sequence
$\left\langle \alpha_{n}\right\rangle _{n\in\omega}$ with limit $\alpha$ such
that each $t_{n}$ is $\alpha_{n}$-uniform in $T.$ \qquad\ \ \ \qquad\ \ 

\item We will say that $T$ is $\alpha$\emph{-uniform }if $rt(T)$ is $\alpha
$-uniform in $p.$
\end{enumerate}
\end{mydef}

Note that not every $s\in T$ must be $\alpha$-uniform for some $\alpha.$ In
the same way, not every $T$ is $\alpha$-uniform for some $\alpha.$ We will
simply say that $T$ is \emph{uniform }if it is $\alpha$-uniform for some
$\alpha<\omega_{1}.$\ We now have the following:

\begin{prop} \label{Katetovequivalent}
Let $T\in\mathbb{W}$ and $1\leq\alpha<\omega_{1}$ such that $T$ is $\alpha
$-uniform. The ideal $\mathcal{F}^{\ast}\left(  T\right)  $ and 
$fin^{\alpha}$ are Kat\v{e}tov equivalent. \label{prop equivalentes}
\end{prop}

\begin{proof}
We proceed by induction on $\alpha.$ It is easy to see that if $T$ is
$1$-uniform, then $\mathcal{F}^{\ast}\left(  T\right)  $ is not tall, so it is
Kat\v{e}tov equivalent to $fin^{1}=\left[  \omega\right]  ^{<\omega
}.$

Assume the Proposition is true for $\alpha,$ we will prove it is true for
$\alpha+1$ as well. Let $T$ be an $\left(  \alpha+1\right)  $-uniform tree and
$succ_{T}\left(  \omega\right)  =\left\{  s_{n}: n\in\omega\right\}  ,$ we
know that each $s_{n}$ is $\alpha$-uniform in $T.$ For every $n\in\omega,$ let
$T^{n}=T_{s_{n}}.$ It follows by the inductive hypothesis that $\mathcal{F}%
^{\ast}\left(  T_{n}\right)  $ and $fin^{\alpha}$ are Kat\v{e}tov
equivalent. Recall that if $A\subseteq L\left(  T\right)  ,$ then
$A\in\mathcal{F}\left(  T\right)  $ if and only if $A\in\mathcal{F}\left(
T^{n}\right)  $ for almost all $n\in\omega.$ In this way, if $A\subseteq
L\left(  T\right)  ,$ we get the following:

\qquad\ \ \ \qquad\ \ 

\hfill%
\begin{tabular}
[c]{lll}%
$A\in\mathcal{F}^{\ast}\left(  T\right)  $ & if and only if & $L\left(
T\right)  \setminus A\in\mathcal{F}\left(  T\right)  $\\
& if and only if & $\forall^{fin^*}n\in\omega\left(  L\left(  T\right)
\setminus A\in\mathcal{F}\left(  T^{n}\right)  \right)  $\\
& if and only if & $\forall^{fin^*}n\in\omega\left(  A\in\mathcal{F}^{\ast
}\left(  T^{n}\right)  \right)  .$%
\end{tabular}
\qquad\ \hfill

\qquad\ \ \ \ \ 

It follows from the inductive hypothesis that $\mathcal{F}^{\ast}\left(
T\right)  $ is Kat\v{e}tov equivalent to $fin^{\alpha+1}.$

Finally, let $\alpha$ be a limit ordinal and assume that the Proposition is true
for every $\beta<\alpha.$ Let $T$ be an $\alpha$-uniform tree. Let $\{  s_{n}: n\in
\omega\}$ be an
enummeration of $succ_{T}(\omega)$ and  $\left\langle \alpha
_{n}\right\rangle _{n\in\omega}$  be an increasing sequence with limit $\alpha$ such that each $s_{n}$ is
$\alpha_{n}$-uniform in $T.$ Define $T^{n}=T_{s_{n}},$ by the inductive
hypothesis we know that $\mathcal{F}^{\ast}\left(  T_{n}\right)  $ and
$fin^{\alpha_{n}}$ are Kat\v{e}tov equivalent. Let $B=\left\{
\alpha_{n}: n\in\omega\right\}  $ and $\overline{B}=%
{\displaystyle\bigcup\limits_{\beta\in B}}
\left(  \left\{  \beta\right\}  \times X_{\beta}\right)  $ (where $X_{\beta}$
was defined in Definition \ref{fin alpha}). Recall that if $A\subseteq
L\left(  T\right)  ,$ then $A\in\mathcal{F}\left(  p\right)  $ if and only if
$A\in\mathcal{F}\left(  T_{n}\right)  $ for almost all $n\in\omega.$ In this
way, if $A\subseteq L\left(  T\right)  ,$ we get the following:

\qquad\ \ \ \qquad\ \ 

\hfill%
\begin{tabular}
[c]{lll}%
$A\in\mathcal{F}^{\ast}\left(  T\right)  $ & if and only if & $L\left(
T\right)  \setminus A\in\mathcal{F}\left(  T\right)  $\\
& if and only if & $\forall^{fin^*}n\in\omega\left(  L\left(  T\right)
\setminus A\in\mathcal{F}\left(  T^{n}\right)  \right)  $\\
& if and only if & $\forall^{fin^*}n\in\omega\left(  A\in\mathcal{F}^{\ast
}\left(  T^{n}\right)  \right)  .$%
\end{tabular}
\qquad\ \hfill

\qquad\ \ \ \ \ 

It follows from the inductive hypothesis that $\mathcal{F}^{\ast}\left(
T\right)  $ is Kat\v{e}tov equivalent to $fin^{\alpha}\upharpoonright
\overline{B}.$ By Lemma \ref{lemma uniforme}, we conclude that $\mathcal{F}%
^{\ast}\left(  T\right)$ and $fin^{\alpha}$ are Kat\v{e}tov equivalent.
\end{proof}

\begin{mydef}
Given $(T,\leq_T)\in \mathbb{W}$, we define recursively, for every $t\in T,$ the rank function as $$\rho_{(T,\leq_T)}(t)=\sup\{\rho_{(T,\leq_T)}(s)+1: t<_T s\}.$$
\end{mydef}
 
It is well known that for every $T\in\mathbb{W}$ the range of $\rho_T$ is an ordinal. Furthermore, $t<_T s$ implies $\rho(s)<\rho(t)$.

As we said before, not every $T$ is uniform, however, it is always possible to prune $T$ to a uniform Laver subtree. 

\begin{prop}
Let $T\in\mathbb{W}.$ There is $S\in\mathbb{L}\left(  T\right)  $ that is
uniform. \label{prop uniformes densos}
\end{prop}
\begin{proof}
By induction on $\rho_T(rt(T))$. If $\rho(rt(t))=0$ there is nothing to do. So, assume that $\rho_T(rt(T))=\alpha>0$. Let $\{s_n:n\in\omega\}$ be an enumeration of $succ_T(rt(T))$ and apply the inductive hypothesis to each $T_{s_n}$ to get $S^n\in\mathbb{L}(T_{s_n})$ that is $\beta_n$-uniform. If there exists $X\in [\omega]^\omega$ such that $\beta_n=\beta_m$ for every $n,m\in X$, then $S=\bigcup_{n\in X} S^n$ is a Laver subtree of $T$ that is $\beta+1$-uniform. In other case, it is possible to find $X\in [\omega]^\omega$ such that $\langle\beta_n\rangle_{n\in X}$ is an increasing sequence with limit $\gamma<\omega_1$. Thus, $S=\bigcup_{n\in X} S^n$ is a Laver subtree of $T$ that is $\gamma$-uniform.
\end{proof}

\begin{cor}
Let $T\in\mathbb{W}.$ There is $\alpha<\omega_{1}$ and $\mathcal{I}$ an ideal
Kat\v{e}tov equivalent to $fin^{\alpha}$ such that $\mathcal{F}%
^{\ast}\left(  T\right)  \subseteq\mathcal{I}.$
\end{cor}

\begin{proof}
By Proposition \ref{prop uniformes densos}, we know that there are
$\alpha<\omega_{1}$\ and\ $S\in\mathbb{L}\left(  T\right)  $ such that $S$
is $\alpha$-uniform. Since $\mathcal{F}\left(  T\right)  $ $\mathcal{\subseteq
F}\left(  S\right)  ,$ the result follows by Proposition
\ref{prop equivalentes}.
\end{proof}

\section{Forcing an Ultrafilter with Trees}

Recall that for each set $x$, the \emph{transitive closure} of $x$ (denoted by $trcl(x)$) is the $\subseteq$-minimal transitive set containing $x$. For any infinite cardinal $\kappa$, define $\textsf{H}(\kappa)=\{x:|trcl(x)|<\kappa\}$. The following is well-known.

\begin{theorem}[see \cite{oldKunen}]
If $\kappa>\omega$ is regular, then  $\textsf{H}(\kappa)$ is a model of $ZFC-P$ ($ZFC$ except Power Set axiom).
\end{theorem}

We now have the following definition.
\begin{mydef}
Define $\mathbb{P}$ as the set of all trees $p\in\mathbb{W}$ such that:

\begin{enumerate}
\item $p\in$ \textsf{H}$\left(  \omega_{1}\right)  $(or any other rich enough set).

\item $L\left(  p\right)  \subseteq\omega$ and if a natural number is in $p,$
then it is a leaf.

\item The root of $p$ is $\omega.$
\end{enumerate}
\end{mydef}

Note that if $p\in\mathbb{P},$ we can view $\mathcal{F}\left(  p\right)  $ as a filter
on $\omega.$

\begin{mydef}\label{definicion de forcing para ultrafiltro Shelah}
Let $p,q\in\mathbb{P}.$ Define $p\leq q$ if there are $p^H\in\mathbb{H}(p)$ and $q^L\in\mathbb{L}(q)$ such that:

\begin{enumerate}

\item $succ_{q^L}(\omega)$ is a barrier in
$p^H$.

\item If $s\in succ_{q^L}(\omega)$ then $[s]_{p^H}=[s]_{q^L}$.

\item If $s\in succ_{q^L}(\omega)$ and $t,t'\in [s]_{p^H}$, then $t\leq_{p^H} t'$ if and only if $t\leq_{q^L} t'$.
\end{enumerate}
\end{mydef}

In the situation above, we say that $(p^H, q^L)$ is a
witness of $p\leq q.$ Note that if $(p^H, q^L)$ is a
witness of $p\leq q$, then $q^L\subseteq p^H$; moreover, if $t\in p^H\backslash q^L$, then there is $s\in succ_{q^L}(\omega)$ such that $t<_p s$. 

Let $q\in\mathbb{P}$ and $q^H\in\mathbb{H}(q)$, then it is clear that $(q^H,q^H )$ is a witness of $q\leq q$ and therefore the relation $\leq$ is reflexive. We will show now that $\leq$ is transitive.

\begin{lemma}
Let $p,q,r\in\mathbb{P}.$ If $p\leq q\leq r,$ then $p\leq r.$
\end{lemma}
\begin{proof}
Let $(p^H,q^L)$ be a witness of $p\leq q$ and $(q^H, r^L)$ be a witness of $q\leq r$. Define $\overline{q}=q^L\cap q^H$. It is clear that $\overline{q}$ is a Laver subtree of $q$; moreover it is a Hechler subtree of $q^L$. Define $\overline{p}$ as the downward closure in $p$ of the set $\bigcup\{[s]_{\overline{q}}: s\in succ_{\overline{q}}(\omega)\}$. On the other hand, define $\overline{r}=\bigcup\{r^L_s\cap \ q^L_s:s\in succ_{r^L}(\omega)\cap\overline{q}\}$. It is easy to see that $\overline{p}$ is a Hechler subtree of $p$ and $\overline{r}$ is a  Laver subtree of $r$. Moreover, ($\overline{p},\overline{q}$) and ($\overline{q},\overline{r}$) witnesses $\overline{p}\leq\overline{q}$ and $\overline{q}\leq\overline{r}$ respectively. The following claim finishes the proof.
\begin{claim}
    $(\overline{p},\overline{r})$ is a witness of $p\leq r$.
\end{claim}
\begin{proof}
We will see first that $succ_{\overline{r}}(\omega)$ is an antichain in $\overline{p}$. Suppose there are $s_0,s_1\in succ_{\overline{r}}(\omega)$ 
that are compatible in $\overline{p}$. Without loss of generality, $s_0<_{\overline{p}} s_1$. Since $succ_{\overline{r}}(\omega)$ is a barrier in $\overline{q}$, there must be $t\in succ_{\overline{q}}(\omega)$ such that $t\leq_{\overline{q}} s_0$. Since ($\overline{p},\overline{q}$) is a witness of $\overline{p}\leq\overline{q}$, then $[t]_{\overline{p}}=[t]_{\overline{q}}$. Thus, we have that $t\leq_{\overline{q}}s_0 <_{\overline{q}} s_1$, which is a contradiction. Now let $X$ be a branch through $\overline{p}$, then $X'= X\cap \overline{q}$ is a branch through $\overline{q}$. Thus, since  $succ_{\overline{r}}(\omega)$ is a barrier in $\overline{q}$, there must be $s\in succ_{\overline{r}}(\omega)$ such that $s\in X'\subseteq X$ and therefore we have that $succ_{\overline{r}}(\omega)$ is a barrier in $\overline{p}$.

Let $s\in succ_{\overline{r}}(\omega)$. Observe that $$[s]_{\overline{r}}=[s]_{r^L}\cap [s]_{q^L}=[s]_{q^H}\cap[s]_{q^L}=[s]_{\overline{q}}.$$ On the other hand, it follows directly from condition (3) of Definition \ref{definicion de forcing para ultrafiltro Shelah} that $$[s]_{\overline{p}}=[s]_{\overline{q}}.$$ Thus, $[s]_{\overline{p}}=[s]_{\overline{q}}$ for every $s\in succ_{\overline{r}}(\omega)$.

To finish the proof let $s\in succ_{\overline{r}}(\omega)$ and fix $t,t'\in[s]_{\overline{p}}$, then we have that $t\leq_{\overline{p}} t'$ if and only if $t\leq_{\overline{q}} t'$ if and only if $t\leq_{\overline{r}} t'$.
\end{proof}
\end{proof}

By Proposition \ref{prop uniformes densos}, we get the following:

\begin{cor}\label{densosuniformes}
If $q\in\mathbb{P}$, then there is $p\in\mathbb{P}$ uniform such that $p\leq
q.$
\end{cor}

Thus, the set of uniform conditions forms a dense subset of $\mathbb{P}$. For some arguments it will be more convenient to restrict oneself to the set of uniform conditions.

\begin{lemma}\label{alphauniform}
Let $\alpha<\omega_{1}$ and $p\in\mathbb{P}$ such that $p$ is $\alpha
$-uniform. For every $\beta\in\omega_{1}$ such that $\beta\geq\alpha,$ there
is $q\leq p$ that is $\beta$-uniform.
\end{lemma}
\begin{proof}
Let $\alpha$ and $p$ be as in the hypothesis. Let $\{P_n:n\in\omega\}$ be a partition of $succ_p(\omega)$ into infinite subsets and  define $S_n=\bigcup\{ p_{s}: s\in P_n\}$ for every $n\in\omega$. Note that each $S_n$ is a Laver subtree of $p$. Furthermore, it is easy to see that each $S_n$ is $\alpha$-uniform.

We will prove the Lemma by induction on $\beta$. If $\beta=\alpha$, there is nothing to do. So, assume that $\beta>\alpha$ and we already proved the result for each $\gamma<\beta$. Suppose that $\beta=\gamma+1$. For every $n\in\omega$, let $S'_n\in\mathbb{P}$ such that the following holds:

(1) $S'_n$ is $\gamma$-uniform

(2) $S'_n\leq S_n$

(3) There is $S_n^L\in\mathbb{L}(S_n)$ such that  $(S'_n, S_n^L)$ is a witness for $S'_n\leq S_n$.\\

Without loss of generality, $S'_n\cap S'_m=\{\omega\}$ whenever $n\neq m$. Let $\{x_n:n\in\omega\}\subseteq H(\omega_1)$ such that $\{x_n:n\in\omega\}\cap S'_n=\emptyset$ for each $S_n'$. For every $n\in\omega$, let $R_n=S'_n\cup\{x_n\}$ and define $\leq_n$ an order on $R_n$ as follows:
\begin{enumerate}[\hspace{0.5cm} (1)]
    \item  $\omega\leq_n x_n\leq_n s$ for every $s\in S'_n\backslash\{\omega\}$.
    \item If $s,t\in S'_n$, then $s\leq_n t$ if and only if $s\leq_{S'_n} t$.
\end{enumerate}  
Define $q=(\bigcup_{n\in\omega} R_n,\bigcup_{n\in\omega} \leq_n)$. Since $succ_q(\omega)=\{x_n:n\in\omega\}$ and each $x_n$ is $\gamma$-uniform in $q$, then it follows that $q$ is $\beta+1$-uniform. It is easy to see that $(q,\bigcup_{n\in\omega} S_n^L)$ is a witness for $q\leq p$.

To finish the proof, suppose that $\beta$ is a limit ordinal. Let $\langle\alpha_n\rangle_{n\in\omega}$ be an increasing sequence with limit $\alpha$ such that $\alpha_0=\alpha$. For every $n\in\omega$, let $S'_n\in\mathbb{P}$ as above, i.e., 

(1) $S'_n$ is $\gamma$-uniform

(2) $S'_n\leq S_n$

(3) There is $S_n^L\in\mathbb{L}(S_n)$ such that  $(S'_n, S_n^L)$ is a witness for $S'_n\leq S_n$.\\

Once again, we can suppose that $S'_n\cap S'_m=\{\omega\}$ whenever $n\neq m$. For every $n\in\omega$, define $(R_n,\leq_n)$ as in the previous case and let $q=(\bigcup_{n\in\omega}R_n, \bigcup_{n\in\omega}\leq_n)$. Note that in this case each $x_n$ is $\alpha_n+1$-uniform and therefore $q$ is $\alpha$-uniform. Finally, by construction $(q,\bigcup_{n\in\omega} S_n^L)$ is a witness for $q\leq p$. 
\end{proof}

\begin{lemma}
Let $\alpha<\omega_1$ and $p\in\mathbb{P}$ be a condition $\alpha$-uniform. If $q\in\mathbb{L}(p)$, then $q$ is $\alpha$-uniform.
\end{lemma}
\begin{proof}
By induction on $\alpha$. If $\alpha=1$ there is nothing to do. If $\alpha=\beta+1$ note that for every $s\in succ_q(\omega)$ we have that $q_s\backslash\{\omega\}$ is a Laver subtree of $p_s\backslash\{\omega\}$; so by inductive hypothesis each $q_s\backslash\{\omega\}$ is $\beta$-uniform and therefore $q$ is $\beta+1$-uniform. Now assume that $\alpha$ is an ordinal limit. Let $\{s_n:n\in\omega\}$ be an enumeration of $succ_p(\omega)$ such that each $s_n$ is $\alpha_n$-uniform and the sequence $\langle \alpha_n:n\in\omega\rangle$ converges to $\alpha$. Once again, since $q_{s_n}\backslash\{\omega\}$ is a Laver subtree of $p_{s_n}\backslash\{\omega\}$ whenever $s_n\in succ_q(\omega)$, then $q_{s_n}\backslash\{\omega\}$ is $\alpha_n$-uniform. Since $\{\alpha<\omega_1:\exists s\in succ_q(\omega)(s\emph{ is }\alpha\emph{-uniform in q})\}$ is an infinte subset of $\langle\alpha_n:n\in\omega\rangle$, then we conclude that $q$ is $\alpha$-uniform. 
\end{proof}

Now we can prove the following.

\begin{lemma}\label{uniformehaciaarriba}
Let $p,q\in\mathbb{P}$ be conditions such that $q$ is $\alpha$-uniform and $p$ is $\beta$-uniform. If $p\leq q$, then $\alpha\leq \beta$.    
\end{lemma}
\begin{proof}
Let $(p^H,q^L)$ be a witness for $p\leq q.$ By the previous lemma we know that $q^L$ is $\alpha$-uniform and $p^H$ is $\beta$-uniform. As we said before, it follows from the definition of the order on $\mathbb{P}$ that $q^L\subseteq p^H$ and therefore it must happen that $\alpha\leq \beta$. 
\end{proof}

The next is a very useful Lemma.

\begin{lemma}\label{witness}
Let $F$ be a finite subset of $\mathbb{P}$ and $p$ be an element of $\mathbb{P}$ such that $p\leq q$ for every $q\in F$. Then there exists $p'\in\mathbb{H}(p)$ such that for every $q\in F$ there is $\overline{q}\in \mathbb{L}(p)$ such that  $(p',q)$ is a witness of $p\leq q$.
\end{lemma}
\begin{proof}
For every $q\in F$ choose $(p_q^H,q^L)$ a witness of $p\leq q$ and define $p'=\bigcap_{q\in F}p^H_q$. Since $F$ is finite, by Lemma \ref{subarboles} we have that $p'\in\mathbb{H}(p)$. On the other hand, for each $q\in F$, define $\overline{q}=\bigcup\{q^L_s\cap p'_s:s\in succ_{q^L}(\omega)\cap p'\}$. It is easy to see that if $q\in F$, then $(p',\overline{q})$ is a witness of $p\leq q$.
\end{proof}

Recall that a forcing notion $(\mathbb{P},\leq)$ is $\sigma$-\emph{closed} if for every $\leq$-decreasing sequence $\{p_n:n\in\omega\},$ there exists $p\in\mathbb{P}$ such that $p\leq p_n$ for every $n\in\omega$. It is easy to see that $\sigma$-closed forcings add no new reals.

\begin{prop}
$\left(  \mathbb{P},\leq\right)  $ is $\sigma$-closed.
\end{prop}

\begin{proof}
Let $\left\{  p_n  : n\in\omega\right\}  \subseteq\mathbb{P}$
be a decreasing sequence. For every $n\in\omega,$ choose $p^{H}_n  \in\mathbb{H}\left(  p_n  \right)  $  such that for
every $m\leq n,$ there is $p^L_{m,n}\in\mathbb{L}(p_m)$ such that $(p^H_n, p^L_{m,n})$ is a witness of $p_n \leq p_m$ (this can be done by Lemma \ref{witness}). We
now do the following:\\
\\
(0) Choose $s_{0}\in succ_{p_0}\left(  \omega\right)  .$
Define $R_{0}=(p_0)_{s_{0}}.$\\
(1) Choose $s_{1}\in succ_{p_1^{H}}\left(  \omega\right)
$ that is incompatible with $s_{0}$. Define $R_{1}=(p^{H}_1)_{s_{1}}.$\\
(2) Choose $s_{2}\in succ_{p^{H}_2}\left(  \omega\right)
$ that is incompatible with $s_{0}$ and $s_{1}$. Define $R_{2}=(p^{H}_2)_{s_{2}}$.\\
$\vdots$
\\
(n+1)  Choose $s_{n+1}\in succ_{p^{H}_ {n+1}}\left(
\omega\right)  $ that is incompatible with $s_{0},s_{1},...s_{n}$. Define
$R_{n+1}=(p^{H}_{n+1})_{s_{n+1}}.$\\
$\vdots$
\\
We now define $q=\bigcup_{n\in\omega}
R_{n}.$ It is clear that $q\in\mathbb{P}$. Finally, fix $m\in\omega$ and define  $\overline{q}=\bigcup_{k>m} R_k$ and $\overline{p}_m=\bigcup\{(p^L_{m,k})_s:\exists k> m (s\in succ^L_{p_{m,k }}(\omega)\cap \overline{q})\}$. It is easy to see that $(\overline{q},\overline{p}_n)$ is a witness of $q\leq p_n.$
\end{proof}

\begin{cor}
$\mathbb{P}$ adds no new reals.   
\end{cor}

\begin{lemma}\label{genfilter}
Let $p,q\in\mathbb{P}.$ If $p\leq q,$ then $\mathcal{F}\left(  q\right)
\subseteq\mathcal{F}\left(  p\right)  .$
\end{lemma}

\begin{proof}
Let $q^{H}\in\mathbb{H}\left(  q\right)  ,$ we need to prove that
$L\left(  q^{H}\right)  \in\mathcal{F}\left(  p\right)  .$ Since $q^{H}$ is a
Hechler subtree of $q,$ we know that $q\leq q^{H},$ so $p\leq q^{H}.$ Let
$p^{H}\in\mathbb{H}(p)$ be a witness for $p\leq q^{H}.$ It
follows that $L\left(  p^{H}\right) \subseteq
L\left(  q^{H}\right)  .$ Since $L\left(  p^{H}\right)  \in\mathcal{F}\left(
p\right)  ,$ it follows that $L\left(  q^{H}\right)  \in\mathcal{F}\left(
p\right)  .$
\end{proof}

If $G\subseteq\mathbb{P}$ is a generic filter, in $V\left[  G\right]  $ we
define $\mathcal{U}_{gen}=%
{\displaystyle\bigcup\limits_{p\in G}}
\mathcal{F}\left(  p\right)  .$

\begin{theorem}
$\mathbb{P}$ forces that $\mathcal{U}_{gen}$ is an ultrafilter.
\end{theorem}

\begin{proof}
It is clear that $\mathcal{U}_{gen}$ is forced to be a filter by Lemma \ref{genfilter}. Since
$\mathbb{P}$ does not add reals, it is enough to prove that for every
$p\in\mathbb{P}$ and $X\subseteq\omega,$ there is $q\leq p$ such that either
$X\in\mathcal{F}\left(  q\right)  $ or $\omega\setminus X\in\mathcal{F}\left(
q\right)  .$

Let $p\in\mathbb{P}$ and $X\subseteq\omega.$ If $X\in\mathcal{F}\left(
p\right)  $ there is nothing to do, so assume that $X\notin\mathcal{F}\left(
p\right)  ,$ so $\omega\setminus X\in\mathcal{F}\left(  p\right)  ^{+}.$ By
the Proposition \ref{dicotomia conjunto}, there is $q\in\mathbb{L}\left(
p\right)  $ such that $L\left(  q\right)  \subseteq\omega\setminus X.$ The
condition $q$ has the desired properties.
\end{proof}

We finish this section by showing that $\mathbb{P}$ forces CH, but first we need some auxiliary results and definitions. All these auxiliary results were worked by B. Balcar, M. Doucha and M. Hrušák. The reader can consult the details in \cite{Basetree}.\\

The \emph{height} of a partial order $(P\leq)$ (denoted by $\mathfrak{h}(P)$) is defined  as the minimal
cardinality of a system of open dense subsets of $P$ such that the intersection of the system is not dense. In other words, $$\mathfrak{h}(P)=\min\{|H|: \forall D\in H(D\emph{ is open dense})\wedge(\bigcap H\emph{ is not dense})\}.$$ The height is a forcing invariant, that means every dense subset of an ordering has the same height. It is well know that for every ordering $P$, we have that $\mathfrak{h}(P)$ is the minimal cardinal $\kappa$ such that forcing with $P$ adds a new subset of $\kappa.$ In particular, forcing with $P$
preserves all cardinals less than $\kappa$.

\begin{prop}
$\mathfrak{h}(\mathbb{P})=\omega_1$.
\end{prop}
\begin{proof}
Observe that since $\mathbb{P}$ is $\sigma$-closed it follows that $\mathfrak{h}(\mathbb{P})\geq\omega_1$. For every $\alpha<\omega_1$ define $D_\alpha=\{p\in\mathcal{\mathbb{P}:\exists\beta>\alpha}(\textit{p is }\beta\textit{-uniform})\}$ and $D'_{\alpha}=\{q\in\mathbb{P}:\exists p\in D_\alpha (q\leq p)\}$. By Corolary \ref{densosuniformes} and Lemma \ref{alphauniform} we have that each $D'_\alpha$ is open dense. The following claim finishes the proof.
\begin{claim}
$\bigcap_{\alpha<\omega_1}D'_\alpha=\emptyset$.
\end{claim}
\begin{proof}
Suppose that $p\in\bigcap_{\alpha<\omega_1}D'_\alpha$. Let $p'$ be a $\gamma$-uniform condition such that $p'\leq p$. Observe that since each $D'_\alpha$ is open, then $p'\in D'_\alpha$ for every $\alpha<\omega_1$. Let $\alpha>\gamma$. Since $p'\in D'_\alpha$, there are $\beta>\alpha$ and a condition $q$ that is $\beta$-uniform such that $q\leq p'$, but then by Lemma \ref{uniformehaciaarriba} we have that $\beta\leq \gamma$ which is a contradiction. Thus, $\bigcap_{\alpha<\omega_1}D'_\alpha=\emptyset$.
\end{proof}
\end{proof}

An ordering $P$ is homogeneous in $\mathfrak{h}$ (homogeneous in height) if for every $p\in\mathbb{P}$ it follows that $\mathfrak{h}(\downarrow p)=\mathfrak{h}(P)$ where $\downarrow p=\{q\in\mathbb{P}:q\leq p\}$.

\begin{lemma}\label{homogeneousinheight}
$\mathbb{P}$ is homogeneous in height.
\end{lemma}
\begin{proof}
Let $p\in \mathbb{P}$. We want to show that $\mathfrak{h}(\downarrow p)=\omega_1$. It is clear that $\mathfrak{h}(\downarrow p)\geq\omega_1$. For every $\alpha<\omega$, define $D'_\alpha$ as in the proof of the above Proposition and let $E_\alpha=D'_\alpha\cap\downarrow p$. The same argument as before shows that $\{E_\alpha:\alpha<\omega_1\}$ is a family of open dense subsets of $\downarrow p$ without dense intersection.
\end{proof}

\begin{mydef}(Base tree property)\label{DefBT}
An ordering $(P,\leq)$ has the base
tree property (we shall shortly say it has the BT-\emph{property}) if it contains
a dense subset $D\subseteq P$ with the following properties:
\begin{enumerate}[\hspace{0.5 cm} (1)]
\item It is atomless, that is, for every $d\in D$ there are $d_1$ and $d_2$ bellow $d$ such that $d_1$ and $d_2$ are incompatible.
\item $D$ is $\sigma$-closed.
\item $|D|\leq\mathfrak{c}$.
\end{enumerate}
\end{mydef}

\begin{prop}\label{PhasBT}
$\mathbb{P}$ has the base
tree property   
\end{prop}
\begin{proof}
Take $D=\mathbb{P}$. It follows directly from the definition of $\mathbb{P}$ that $D$ satisfies all the conditions in the Definition \ref{DefBT}.
\end{proof}

\begin{theorem}[The base tree Theorem, see \cite{Basetree}]
Let $(P,\leq)$ be an ordering
homogeneous in the height with the BT-property. Then there exists a sequence $\{T_\alpha:\alpha<\mathfrak{h}(P)\}$ of maximal antichains of $P$ such that:
\begin{enumerate}[\hspace{0.5cm} (1)]
    \item $T_P=\bigcup_{\alpha<\mathfrak{h}(P)}T_\alpha$ is a tree of height $\mathfrak{h}(P)$, where $T_\alpha$ is the $\alpha$-th level of the tree.
    \item Each $t\in T$ has $\mathfrak{c}$ immediate successors.
    \item $T$ is dense in $P$.
\end{enumerate}
$T$ is called the base tree of $P$.
\end{theorem}

Lemma \ref{homogeneousinheight} and Proposition \ref{PhasBT} show that $\mathbb{P}$ satisfies the conditions of the base tree Theorem. Note that clause (3) implies that forcing with $T_\mathbb{P}$ is equivalent to forcing with $\mathbb{P}$. 

\begin{theorem}
$T_\mathbb{P}$ forces CH. In particular, $\mathbb{P}$ forces CH.
\end{theorem}
\begin{proof}
By clause (2) it is easy to see that $T_\mathbb{P}$ adds a sujective function from $\mathfrak{h}(\mathbb{P})$ (which is equal to $\omega_1$) to $\mathfrak{c}$. Since $\mathfrak{h}(\mathbb{P})$ is not collapsed and $T_\mathbb{P}$ adds no new reals, then $\mathfrak{c}^{V[G]}=\mathfrak{c}^V\leq\mathfrak{h}(\mathbb{P})$. Thus, $V[G]\models \omega_1=\mathfrak{h}(\mathbb{P})=\mathfrak{c}$ for every $T_\mathbb{P}$-generic filter $G$.
\end{proof}

\section{Shelah ultrafilters}

As we saw in the previous section, $\mathcal{U}_{gen}$ is forced to be an ultrafilter on $\omega$. However, $\mathcal{U}_{gen}$ satisfies some critical properties which we summarize in the following definition.

\begin{mydef}
Let $\mathcal{U}$ be an ultrafilter on $\omega$. We say that $\mathcal{U}$ is a Shelah ultrafilter if the following holds:
\begin{enumerate}
    \item For every $\alpha<\omega_1$, there is and ideal $\mathcal{I}$ isomorphic to $fin^\alpha$ such that $\mathcal{I}\subseteq\mathcal{U}^*$.
    \item Let $\mathcal{I}$ be an analytic ideal. If $\mathcal{I}\cap\mathcal{U}=\emptyset$, then there are $\alpha<\omega _1$  and an ideal $\mathcal{J}$ such that $\mathcal{J}$ is Kat\v{e}tov equivalent to $fin^\alpha$ and $\mathcal{I}\subseteq\mathcal{J}\subseteq\mathcal{U}^*$.
\end{enumerate}
\end{mydef}

\begin{theorem}
$\mathbb{P}$ forces that $\mathcal{U}_{gen}$ is a Shelah ultrafilter.
\end{theorem}
\begin{proof}
Let $\alpha\leq\omega_1$ and $D_\alpha=\{p\in\mathbb{P}:\exists\beta\geq\alpha (\textit{p is }\beta\textit{-uniform})\}$. It follows from Corollary \ref{densosuniformes} and  Lemma \ref{alphauniform} that $D_\alpha$ is dense. Let $p\in D_\alpha$, then by Proposition \ref{Katetovequivalent}, $\mathcal{F}(p)^*$ is Kat\v{e}tov equivalent to $fin^\beta$ for some $\beta\geq\alpha$. By Corollary \ref{copiaisomorfa}, we have that $\mathcal{F}(p)$ contains an isomorphic copy of $fin^\alpha$. Thus, there exists an ideal $\mathcal{I}$ isomorphic to $fin^\alpha$ such that $\mathcal{I}\subseteq\mathcal{F}(p)$ and  therefore, $p$ forces that $\mathcal{I}\subseteq\mathcal{U}_{gen}^*$.

Let $\Dot{\mathcal{I}}$ be a name for an analytic ideal and $p$ be a condition that forces that $\Dot{\mathcal{I}}\cap\Dot{\mathcal{U}}_{gen}=\emptyset$. Let $I\in V$ and assume that $I\in \mathcal{F}(p)^+$, then there exists $q\leq p$ such that $L(q)\subseteq I$ and therefore $q$ forces $I\in\Dot{\mathcal{U}}_{gen}$, which is a contradiction. Thus, $p$ forces $\Dot{\mathcal{I}}$ to be contained in $\mathcal{F}^(p)^*$.   
Finally, let $q\leq p$ that is $\alpha$ uniform for some $\alpha\in\omega_1$. Then, by Proposition \ref{Katetovequivalent}, $\mathcal{F}^*(q)$ is Kat\v{e}tov equivalent to $fin^\alpha$ and $q$ forces $\Dot{\mathcal{I}}$ to be contained in $\mathcal{F}^*(q)\subseteq\mathcal{U}^*_{gen}$.
\end{proof}

At this point it is worth remembering the definitions of the following special ultrafilters.

\begin{mydef}
Let $\mathcal{U}$ be a non-principal ultrafilter on $\omega$.
\begin{enumerate}[\hspace{0.5cm} (1)]
    \item We say that $\mathcal{U}$ is P-point if for every $\{U_n:n\in\omega\}\subseteq \mathcal{U}$ there exists $U\in\mathcal{U}$ such that $U\subseteq^* U_n$ for every $n\in\omega$.
    \item We say that $\mathcal{U}$ is Q-point if for each partition on $\omega$ into finite sets $\{I_n:n\in\omega\}$, there exists $U\in\mathcal{U}$ such that $|U\cap I_n|\leq 1$ for every $n\in\omega$. 
    \item We say that $\mathcal{U}$ is \emph{Ramsey} if for every $c:[\omega]^2\longrightarrow 2$ there exists $U\in\mathcal{U}$ such that $U$ is monochromatic for $c$, i.e., $c\upharpoonright[U]^2$ is constant.
\end{enumerate}
\end{mydef}

It is well know that an ultrafilter $\mathcal{U}$ is Ramsey if and only if it is a P-point and a Q-point (see \cite{Barty}).  We will see later that Shelah ultrafilters cannot be P-points and therefore, they cannot be Ramsey. However, the definition of  Shelah ultrafilter turns out to be similar to the following importants caracterizations of P-points and Ramsey ultrafilters. 

\begin{theorem}[Mathias \cite{Happyfamilies}]
Let $\mathcal{U}$ be an ultrafilter on $\omega$. Then $\mathcal{U}$ is Ramsey if and only if $\mathcal{U}\cap\mathcal{I}\neq\emptyset$ for every analytic tall ideal $\mathcal{I}.$
\end{theorem}

\begin{theorem}[Zapletal \cite{PreservingP-points}]
Let $\mathcal{U}$ be an ultrafilter on $\omega$. The following are equivalent:
\begin{enumerate}
    \item $\mathcal{U}$ is P-point.
    \item If $\mathcal{I}$ is an analytic ideal such that $\mathcal{U}\cap\mathcal{I}=\emptyset$, then there exists $\mathcal{J}$ a $F_\sigma$ ideal such that $\mathcal{I}\subseteq\mathcal{J}$ and $\mathcal{J}\cap \mathcal{U}=\emptyset$.
\end{enumerate}
\end{theorem}

For the rest of the paper we will study properties of $\mathcal{U}_{gen}$ and Shelah ultrafilters.

Recall that $\textit{Sacks forcing}$ (denoted by $\mathbb{S}$) is the set of perfect subtrees of $(2^{<\omega},\subseteq)$ ordered by inclusion. This forcing adds a generic real defined as $g = \bigcap\{[p]:p\in G\}$ where $G$ is the generic filter on $\mathbb{S}$ (see \cite{Sacks}). In \cite{YiparakiThesis}, Yiparaki defined a special type of reaping families\footnote{Recall that a set $\mathcal{R}\subseteq[\omega]^\omega$ is a \emph{reaping family} if for every $A\in[\omega]^\omega$ there exists $R\in\mathcal{R}$ such that either $R\subseteq A$ or $R\cap A=\emptyset.$} that are essential to characterize the indestructibility of filters\footnote{Given a forcing notion $\mathbb{P}$, an ultrafilter $\mathcal{U}$ is $\mathbb{P}$-indestructible if $\mathcal{U}$ generates an ultrafilter in every extension by $\mathbb{P}.$} by Sacks forcing.

\begin{mydef}[Yiparaki \cite{YiparakiThesis}]
A family $\mathcal{R}\subseteq [\omega]^{\omega}$ is called a Halpern-L\"auchli if for every $c: 2^{<\omega}\rightarrow 2$ there are $p\in\mathbb{S}$ and $A\in \mathcal{R}$ such that $c$ is constant on $p\upharpoonright A$.
\end{mydef}

The importance of Halpern-L\"auchli families lies on the following result:

\begin{theorem} (see \cite{HLideals})
Let $\mathcal{U}$ be an ultrafilter on $\omega$. Then $\mathcal{U}$ is Sacks indestructible if and only if it is a Halpern–Läuchli family. 
\end{theorem}

If $\mathcal{I}$ is an ideal, we will say that $\mathcal{I}$ is HL as a shortcut for the statement that $\mathcal{I}^+$ is  a Halpern–Läuchli family. It is easy to see that if $\mathcal{I,J}$ are ideals such that $\mathcal{I}\leq_K\mathcal{J}$ and $\mathcal{J}$ is HL, then $\mathcal{I}$ is HL. Given $c:2^{<\omega}\longrightarrow 2$, for each $p\in\mathbb{S}$ define $$H_c(p)=\{n\in\omega:\textit{c is constant on }p\upharpoonright\{n\}\}.$$ Let $\mathcal{I}_c$ be the (possibly improper) ideal generated by $\{H_c(p):p\in\mathbb{S}\}.$ It is easy to see that a set $\mathcal{R}$ is not a Halpern-L\"auchli family if and only if there exists $c:2^{<\omega}\longrightarrow 2$ such that $\mathcal{I}_c\subseteq\mathcal{R}$.

\begin{lemma}[\cite{GruffandScat}]
$fin^{\alpha}$ is an HL ideal for every $\alpha<\omega_1$.
\end{lemma}

\begin{cor}
Shelah ultrafilters are $\mathbb{S}$-indestructible.
\end{cor}
\begin{proof}
Let $\mathcal{U}$ be a Shelah ultrafilter and  suppose that there exists $c:2^{<\omega}\longrightarrow 2$ such that $\mathcal{I}_c\subseteq\mathcal{U}^*.$ Then by Definition there are $\alpha<\omega _1$  and an ideal $\mathcal{J}$ such that $\mathcal{J}$ is Kat\v{e}tov equivalent to $fin^\alpha$ and $\mathcal{I}_c\subseteq\mathcal{J}\subseteq\mathcal{U}^*$. Since $fin^\alpha$ is HL and $\mathcal{I}_c\leq_K fin^\alpha$, then $\mathcal{I_c}$ is HL, which is a contradiction. 
\end{proof}

Recall that in the usual monochromatic Ramsey theory one is given a coloring $c:[\omega]^n\longrightarrow\omega$ and seeks a set $A\in[\omega]^\omega$ which is \emph{monochromatic} for $c$, i.e., there is a single color which all elements of $[A]^n$ receive. In the polychromatic Ramsey theory we instead seek a set $B\in[\omega]^\omega$ which is \emph{polychromatic} for $c$, i.e., each member of $[B]^n$ receives a different color. Of course, in order to find monochromatic or polychromatic sets we need to put some restrictions on the colorings considered. Thus, when we are talking about polychromatic Ramsey theory we only consider colorings in which every color is used a finite number of times. To learn more about  polychromatic Ramsey theory the reader can consult \cite{PolychromaticRamsey} and \cite{COMPARISONSOFPOLYCHROMATICANDMONOCHROMATIC}.

\begin{mydef}
Let $c:[\omega]^2\longrightarrow\omega$ and $k\in\omega$. We say that $c$ is $k$-bounded if $|c^{-1}[n]|\leq k$ for every $n\in\omega$.
\end{mydef}

As we mentioned before, an ultrafilter
$\mathcal{U}$ is Ramsey if for every $c:[\omega]^2\longrightarrow 2$ there exists $A\in\mathcal{U}$ such that $A$ is monochromatic for $c.$ The \textquotedblleft polychromatic version\textquotedblright\ of Ramsey ultrafilter was introduced by Palumbo in \cite{COMPARISONSOFPOLYCHROMATICANDMONOCHROMATIC} and it is called \emph{rainbow Ramsey ultrafilter}. 

\begin{mydef}[Palumbo]
A nonprincipal ultrafilter $\mathcal{U}$ is rainbow Ramsey if for every 2-bounded coloring $c:[\omega]^2\longrightarrow\omega$ there is an $A\in\mathcal{U}$ polychromatic for $c$. 
\end{mydef}

\begin{prop}
Let $1<\alpha<\omega_1$, $T\in\mathbb{W}$ that is $\alpha$-uniform, $k\in\omega$ and $c:[L(T)]^2\longrightarrow\omega$ that is a k-bounded coloring. Then there is $S\in\mathbb{L}(T)$ such that $L(S)$ is polychromatic for $c$.
\end{prop}
\begin{proof}

Let $\{s_n:n\in\omega\}$ be an enumeration of $X=\{s\in T: s\textit{ is 1-uniform}\}$. We will recursively construct $\{F_n:n\in\omega\}$ a sequence of finite subsets of $L(T)$ such that for every $n\in\omega$ the following holds:
\begin{enumerate}[\hspace{0.5cm} (1)]
\item $F_n\subseteq F_{n+1}$.
\item $F_n$ is polychromatic.
\end{enumerate}

We start by choosing $t\in succ_T(s_0)$ and $r\in succ_T(s_1)$. Define $F_0=\{t,r\}$. Now suppose we already defined $F_n$. For every $s_j$ with $j\leq n+1$ choose $t_j\in succ_T(s_j)\backslash F_n$ such that $F_n\cup\{t_j:j\leq n+1\}$ is polychromatic. Note that this can be done because $c$ is $k$-bounded and $succ_T(s)$ is infinite for every $s\in X$. Finally, let S be the downwards clousure of $F=\bigcup_{n\in\omega}F_n$. It is easy to see that $S\in\mathbb{L}(T)$ and $L(S)=F$ is polychromatic for $c$.
\end{proof}

\begin{prop}\label{U_genisrainbowRamsey}
$\mathbb{P}$ forces that $\mathcal{U}_{gen}$ is a rainbow Ramsey ultrafilter.
\end{prop}
\begin{proof}
Since $\mathbb{P}$ does not add new reals we only need to consider colorings from $V$. Let $p\in\mathbb{P}$ and $c:[L(T)]^2\longrightarrow\omega$ that is a 2-bounded coloring. Now let $r\in\mathbb{L}(q)$ such that $L(r)$ is polychromatic for $c$. Then $r$ forces that $L(r)\in\Dot{\mathcal{U}}_{gen}$ and $L(r)$ is polychromatic for $c$.    
\end{proof}

\section{\texorpdfstring{$\mathcal{I}$-ultrafilters}{I-ultrafilters}}

We start the section by enunciating the next definition introduced by J. Baumgartner in \cite{RefBaumgartner}.

\begin{mydef}\label{Iultrafilters}
Let $\mathcal{I}$ an ideal on $X$ and $\mathcal{U}$ be an ultrafilter on $Y$. We say that $\mathcal{U}$ is a $\mathcal{I}\textit{-ultrafilter}$ if for every $f:Y\longrightarrow X$ there is $U\in\mathcal{U}$ such that $f[U]\in\mathcal{I}$. Equivalently, $\mathcal{U}$ is a $\mathcal{I}\textit{-ultrafilter}$ if and only if $\mathcal{I}\nleq_K\mathcal{U}^*$.    
\end{mydef}

It follows directly from the Definition that if $\mathcal{U}$ is a $\mathcal{I}$-ultrafilter and $\mathcal{I}\leq_K\mathcal{J}$, then $\mathcal{U}$ is a $\mathcal{J}$-ultrafilter. The importance of $\mathcal{I}$-ultrafilters lies in the fact that we can characterize many standard combinatorial properties of ultrafilters in
this way via Borel ideals (see \cite{FIN^2ULTRAFILTERS}). 

We now list some classical Borel ideals mentioned in this section.

\begin{enumerate}[\hspace{0.5cm}(1)]
\item $nwd=\{A\subseteq\mathbb{Q}: A\text{ is nowhere dense}\}.$

\item  $conv$ the ideal on $\mathbb{Q}\cap [0,1]$ generated by sequences in $\mathbb{Q}\cap [0,1]$ convergent in $[0,1]$. 

\item $\mathcal{ED}=\{A\subseteq\omega\times\omega: (\exists m,n\in\omega)(\forall k>n)(|\{l\in\omega: (k,l)\in A\}|<m)\}.$

\item $\mathcal{ED}_{fin}=\mathcal{ED}\upharpoonright\Delta$, where $\Delta=\{(n,m): m\leq n\}$.

\item  $\mathcal{I}_{\frac{1}{n}}=\{A\subseteq\omega:\sum_{n\in A}\frac{1}{n+1}<\infty\}$.

\item  $\mathcal{Z}=\{A\subseteq \omega: lim_{n\rightarrow \infty} \frac{|A\cap n|}{n}=0\}.$

\item The $\textit{Solecki ideal}$ $\mathcal{S}$  is the ideal on the set $\Omega=\{A\in Clop(2^\omega):\lambda(A)\leq 1/2\}$ (where $\lambda$ is the standard Haar measure on $2^\omega$), generated by the sets of the form $I_x=\{A\in\Omega:x\in A\}$ ($x\in 2^\omega$).
\end{enumerate}

The Kat\v{e}tov relationships between the above-mentioned ideals are summarized in the following diagram.

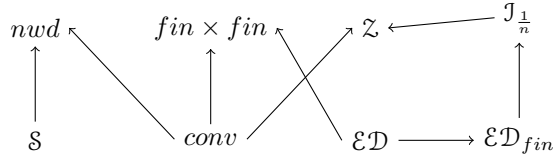
\begin{center}
\begin{tikzpicture}

\node (nwd)   {$nwd$};

\node (finxfin) [right=of nwd] {$fin \times fin$};

\node (Z) [right=of finxfin] {$\mathcal{Z}$};

\node (S) [below=of nwd] {$\mathcal{S}$};

\node (conv) [below=of finxfin] {$conv$};

\node (ED) [below=of Z] {$\mathcal{ED}$};

\node (EDFIN) [right=of ED] {$\mathcal{ED}_{fin}$};

\node (sumable) [above=of EDFIN] {$\mathcal{I} _{\frac{1}{n}}$};

\draw[->] (S.north)--(nwd.south) ;

\draw[->] (conv.west) -- (nwd.east);

\draw[->] (conv.north) -- (finxfin.south);

\draw[->] (conv.east) -- (Z.west);

\draw[->] (ED.west) -- (finxfin.east);

\draw[->] (ED.east) -- (EDFIN.west);

\draw[->] (EDFIN.north) -- (sumable.south);

\draw[->] (sumable.west) -- (Z.east);

\end{tikzpicture}
\end{center} 

To exemplify the importance of Definition \ref{Iultrafilters},
 consider the following important
variation of the Kat\v{e}tov order; the \emph{Kat\v{e}tov-Blass} order. Given ideals $\mathcal{I}$, $\mathcal{J}$ on countable sets $X$ and $Y$ respectively, we say that $\mathcal{I}\leq_{KB}\mathcal{J}$ if there exists a finite-to-one function $f:Y\longrightarrow X$ such that $f^{-1}[I]\in\mathcal{J}$ for each $I\in\mathcal{I}$. Clearly $\mathcal{I}\leq_{KB}\mathcal{J}$ implies $\mathcal{I}\leq_{K}\mathcal{J}$.
 In \cite{FIN^2ULTRAFILTERS} it was shown the next result.

\begin{prop}\label{P-pointandfin}
Let $\mathcal{U}$ be an ultrafilter on $\omega$.
\begin{enumerate}[\hspace{0.5cm}(1)]
    \item $\mathcal{U}$ is Ramsey if and only if $\mathcal{U}$ is $\mathcal{ED}$-ultrafilter.
    \item $\mathcal{U}$ is P-point if and only if $\mathcal{U}$ is $fin^2$-ultrafilter if and only if $\mathcal{U}$ is $conv$-ultrafilter.
    \item $\mathcal{U}$ is Q-point if and only if $\mathcal{ED}_{fin}\nleq_{KB}\mathcal{U}$.  
\end{enumerate}  
\end{prop}

It is clear that Shelah ultrafilters cannot be neither $conv$-ultrafilters nor $fin^\alpha$-ultrafilters for any $\alpha<\omega_1$. So, by Proposition \ref{P-pointandfin} we can conclude the following.

\begin{cor}\label{ShelahisnotP-point}
If $\mathcal{U}$ is a Shelah ultrafilter, then it is not a P-point.
\end{cor}

We now want to see for which ideals $\mathcal{I}$ it holds that a Shelah ultrafilter is also a $\mathcal{I}$-ultrafilter. The following Lemma will be very useful to do this.

\begin{lemma}
 Let $\mathcal{U}$ be a Shelah ultrafilter and  $\mathcal{I}$ be an analytic ideal on $\omega$. Then $\mathcal{I}\leq_K\mathcal{U}^*$ if and only if there exists $\alpha<\omega_1$ such that  $\mathcal{I}\leq_K fin^\alpha$.
\end{lemma}
\begin{proof}
Let $f:\omega\longrightarrow\omega$ such that $f^{-1}[I]\in\mathcal{U}^*$ for every $I\in\mathcal{I}$. Let $\mathcal{K}$ be the set $\{f^{-1}[I]: I\in\mathcal{I}\}$. Then  $\mathcal{K}$ is an analytic ideal included in $\mathcal{U}^*$ and hence there are $\alpha<\omega_1$ and an ideal $\mathcal{J}$, such that $\mathcal{J}$ is Kat\v{e}tov equivalent to $fin^\alpha$ and $\mathcal{K}\subseteq\mathcal{J}\subseteq\mathcal{U}^*$. Thus, it follows that $\mathcal{I}\leq_K\mathcal{K}\leq_K\mathcal{J}\leq_K fin^\alpha$.

Now suppose that there exists $\alpha<\omega_1$ such that  $\mathcal{I}\leq_K fin^\alpha$. Let $\mathcal{J}$ be an ideal Kat\v{e}tov equivalent to $fin^\alpha$ such that  $\mathcal{J}\subseteq\mathcal{U}^*$. Thus, $\mathcal{I}\leq_K\mathcal{J}$ and any witness for $\mathcal{I}\leq_K\mathcal{J}$ is also a witness for $\mathcal{I}\leq_K\mathcal{U}^*$.
\end{proof}
 
Thus, knowing for which ideals $\mathcal{I}$ it is true that a Shelah ultrafilter is also a $\mathcal{I}$-ultrafilter, is equivalent to knowing the Kat\v{e}tov relationship between $\mathcal{I}$ and the ideals $fin^\alpha.$

\begin{mydef}
Let $T$ be an element of $\mathbb{W}$. We say that $\bigstar(T)$ holds if for every $f:L(T)\longrightarrow\omega$ one of the following happens: 
\begin{enumerate} [\hspace{0.5cm} (I)]
    \item $*(T,f):=$ There exists $S\in \mathbb{L}(T)$ such that $|f[L(S)]|=1$.
    \item $\bigtriangleup(T,f):=$ There are $S\in \mathbb{L}(T)$ and $A\subset S$ a barrier in $S$ such that for every $t\in A$ and every $s,s'\in succ_S(t)$ ($s\neq s'$) the following holds:
\begin{enumerate}[\hspace{0.8 cm} (a)]
    \item $|f[L(S_s)]|=1$.
    \item $f[L(S_s)]\neq f[L(S_{s'})]$.
\end{enumerate} 
\end{enumerate}
\end{mydef}

\begin{prop}\label{1-uniform}
Let $T$ be 1-uniform. Then $\bigstar(T).$
\end{prop}
\begin{proof}
Let $f$ be a function from $L(T)$ to $\omega$. If there exists $X\in [L(T)]^\omega$ such that $|f[X]|=1$, then the downwards closure of $X$ a witness of $*(T,f)$. In other case, it is posible to construct $X\in [L(T)]^\omega$ such that $f\upharpoonright X$ is injective. Thus, if $A=\{rt(T)\}$ and $S$ is the downwards closure of $X$, then $A$ and $S$ are witnesses to $\bigtriangleup(T,f)$. 
\end{proof}

\begin{prop}\label{generalcase}
Consider $\alpha<\omega_1.$ and $T\in\mathbb{W}$ that is $\alpha$-uniform. Suppose that for every $\beta<\omega_1$ and every $S\in\mathbb{W}$ that is $\beta$-uniform, $\bigstar(S)$ holds. Then $\bigstar(T)$ holds.
\end{prop}
\begin{proof}
Let $f:L(T)\longrightarrow\omega$ and let $\{s_n:n\in\omega\}$ be a enumeration of $succ_T(rt(T))$. Note that since $T$ is $\alpha$-uniform, each $T_{s_n}$ is $\alpha_n$-uniform for some $\alpha_n<\alpha$ and therefore $\bigstar(T_{s_n})$ holds. Define $X=\{n\in\omega: \bigtriangleup(T_{s_n},f) \}$. We first assume that $X$ is infinite. For every $n\in X$, choose $S_n\in\mathbb{L}(T_{s_n})$  and a barrier $A_n\subset T_{s_n}$ that satisfy $\bigtriangleup(T_{s_n},f)$. Then, $S=\bigcup_{n\in X} S_n$ and $A=\bigcup_{n\in X} A_n$ satisfy $\bigtriangleup(T,f)$. Now suppose that $X$ is finite. Without lost of generality $X=\emptyset.$ For every $n\in\omega$, choose $S_n\in\mathbb{L}(T_{s_n})$ that satisfies  $*(T_{s_n},f)$ and define $g:succ_T(rt(T))\longrightarrow\omega$ by $g(s_n)=f[S_{s_n}]$. Consider $R=\{rt(T)\}\cup\{s_n:n\in\omega\}$. Then $R$ is an element of $\mathbb{W}$ that is 1-uniform and therefore $\bigstar(R)$. If there exists $X\in[\omega]^\omega$ such that $|g[\{s_n:n\in X\}]|=1$, then $S=\bigcup_{n\in X}S_n$ is a Laver subtree of $T$ such that $|f[L(S)]|=1$. In other case, construct $X\in[\omega]^\omega$ such that $g\upharpoonright \{s_n:n\in X\}$ is injective. Thus, $S=\bigcup_{n\in X}S_n$ and $A=\{rt(T)\}$ satisfy $\bigtriangleup(T,f)$.
\end{proof}

The following result is an immediate consequence of Proposition \ref{1-uniform} and Proposition \ref{generalcase}.
\begin{theorem}
Let $\alpha<\omega_1$ and let $T$ be $\alpha$-uniform. Then $\bigstar(T)$.  
\end{theorem}

With the previous combinatorial properties of the $\alpha$-uniform trees we can now prove the following.

\begin{prop}
Let $\alpha<\omega_1$,  $T\in\mathbb{W}$ that is $\alpha$-uniform and  $f:L(T)\longrightarrow\Delta$. Then there exists $R\in\mathbb{L}(T)$ such that $f[L(R)]\in\mathcal{ED}_{fin}$.    
\end{prop}
\begin{proof}
Since $T$ is $\alpha$-uniform, then by Proposition \ref{generalcase} we have that $\bigstar(T)$ holds. If $*(T,f)$, then we are done. Assume that $\bigtriangleup(T,f)$. Let $S\in\mathbb{L}(T)$ and $A\subset T$ that satisfies $\bigtriangleup(T,f).$ Let $\{s_n:n\in\omega\}$ be an enumeration of $A$. We will recursively construct $\{k_n:n\in\omega\}$ a sequence of natural numbers and $\{B_n:n\in\omega\}$ a sequence of finite subsets of $S$ such that for every $n,m\in\omega$ the following happens:

(1) $k_n<k_m$ whenever $n<m$.

(2) $f[B_n]\cap (\omega\backslash k_m\times\omega)=\emptyset$ whenever $n\leq m$.

(3) $B_n\cap B_m=\emptyset$ whenever $n\neq m.$

We start by choosing $t\in succ_S(s_0)$. Let $B_0=\{t\}$ and $k_0=\pi[f[L(S_{t})]]+1$ where $\pi$ is the projection on the first coordinate. Note that $n_0$ is well defined because $\bigtriangleup(T,f)$. Now suppose that we already constructed $B_n$ and $k_n$. For every $s_i$ with $i\leq n+1$ choose $t_i\in succ_S(s_i)$ such that $\pi[f[L(S_{t_i})]]> k_n$ and $\pi[f[L(S_{t_i})]]\neq \pi[f[L(S_{t_j})]]$ whenever $i\neq j$. This can be done because of  $(a)$ and $(b)$ in the Definition of $\bigtriangleup(T,f)$. Let $B_{n+1}=\{t_i:i\leq n+1\}$ and define $k_{n+1}=\max\{\pi[f[L(S_{t_i})]:i\leq n+1\}+1$. This finish the construction of $\{k_n:n\in\omega\}$ and $\{B_n:n\in\omega\}$. Let $B=\bigcup_{n\in\omega}B_n$ and define $R\subseteq T$ as the downwards clousure of $\bigcup_{t\in B} S_t$. Note that for every $s\in A$ there are $\omega$ many elements of $R$ that contain $s$. Thus, $A\subset R$ and therefore $R$ is a Laver subtree of $T$. Finally, note that by construction we have that $|f[L(R)]\cap(\{n\}\times\omega)|\leq 1$ so, $f[L(R)]\in\mathcal{ED}_{fin}$.
\end{proof}
 
\begin{prop}
$\mathcal{ED}_{fin}\nleq_K fin^\alpha$ for every $\alpha<\omega_1$. In particular, every Shelah ultrafilter is a $\mathcal{ED}_{fin}$-ultrafilter. 
\end{prop}

\begin{cor}
Every Shelah ultrafilter is Q-point. 
\end{cor}

We now turn our attention to the right side of the diagram. 

Recall that an ideal $\mathcal{I}$ satisfies the $\textit{Fubini property}$ (see \cite{RadonIdeals}) if for any $A$ Borel subset of $\omega\times2^\omega$ and any $\epsilon>0,$ if $\{n\in\omega:\lambda^*(A_n)>\epsilon\}\in\mathcal{I}^+$, then we have that $\lambda^*(\{x\in 2^\omega: A^x\in\mathcal{I}^+\})>\epsilon$ where $A_n=\{x\in2^\omega:(n,x)\in A\}$, $A^x=\{n\in\omega:(n,x\in A)\}$ and $\lambda^*$ denotes the outer Lebesgue measure on $2^\omega$. 

In Theorem 2.1 of \cite{Solecki1}, it was indirectly shown that Solecki ideal is critical for ideals satisfying the Fubini property. The following Theorem is just a consequence of Solecki's result. 

\begin{theorem}
An ideal $\mathcal{I}$ fails to satisfy the Fubini property if
and only if there is an $\mathcal{I}$-positive set $X$ such that $\mathcal{S}\leq_K\mathcal{I}\upharpoonright X$.
\end{theorem}

The reader can consult \cite{FubinipropMichalyDavid} for a complete proof of the previous Theorem. 

\begin{lemma}\label{finFubiniproperty}
$fin$ and $bnd(\alpha)$ ($\omega<\alpha<\omega_1$) have the Fubini property.
\end{lemma}
\begin{proof}
Since $\mathcal{S}$ is tall and $fin$ is Kat\v{e}tov uniform, then $\mathcal{S}\nleq_K fin\upharpoonright X$ for every $X\in fin^+$. On the other hand, it is well known that $bnd(\alpha)$ is Kat\v{e}tov uniform and it is not tall. Thus, $fin\cong_K bnd(\alpha)\upharpoonright A$ for every $A\in bnd(\alpha)^+$ and therefore we have that $\mathcal{S}\nleq_K bnd(\alpha)\upharpoonright A$ for every $A\in bnd(\alpha)^+$.
\end{proof}

\begin{theorem}[Kanovei and Reeken \cite{RadonIdeals}]\label{productFubiniproperty}
Assume that $\mathcal{I}$ is an ideal on a countable set $I$ and $\mathcal{J}_i$ is an ideal on a countable set $J_i$ for any $i\in I$. If $\mathcal{I}$ and every $\mathcal{J}_i$ have the Fubini property, then $\lim_{i\rightarrow\mathcal{I}}\mathcal{J}_i$ has the Fubini property.
\end{theorem}

In particular, the Fubini product of two ideals with the Fubini property has the Fubini property. Thus, as an immediate consequence of Lemma \ref{finFubiniproperty} and Theorem \ref{productFubiniproperty} we have the next result.

\begin{theorem}
$fin^\alpha$ has the Fubini property for any $\alpha<\omega_1$. In particular, we have that  $\mathcal{S}\nleq_K fin^\alpha$ for any $\alpha<\omega_1$.
\end{theorem}

We summarize the most important results of the section in the following Theorem.

\begin{theorem}
Let $\mathcal{U}$ be a Shelah ultrafilter and let $\mathcal{I}\in\{\mathcal{S}, nwd, \mathcal{ED}_{fin},\mathcal{I}_{\frac{1}{n}},\mathcal{Z}\}$. Then $\mathcal{U}$ is a $\mathcal{I}$-ultrafilter.
\end{theorem}

\section{Questions}

In this last section, we will state some problems that we do not know how to solve. The following question is motivated by Proposition \ref{U_genisrainbowRamsey}.

\begin{question}
Is any Shelah ultrafilter a rainbow Ramsey ultrafilter?
\end{question}
\quad
We know that $\mathcal{U}_{gen}$ is a rainbow Ramsey ultrafilter but we do not know if this is the case for every Shelah ultrafilter.\\

Let $D$ and $E$ be two directed orders. We say that a function $f:E\longrightarrow D$ is $\textit{cofinal}$ if the image of each cofinal subset of $E$ is cofinal in $D$. We say that $D$ is $\textit{Tukey reducible}$ to $E$ ($D\leq_T E$) if there exists a cofinal map from $E$ to $D$. If both $D\leq_T E$ and $E\leq_T D$, then we write $E\equiv_T D$ and say that $D$ and $D$ are $\textit{Tukey equivalent}$. $\equiv_T$ is an equivalence relation and $\leq_T$ on the equivalence classes forms a partial ordering.
The equivalence classes can be called $\textit{Tukey types}$ or $\textit{Tukey degrees}$.

It is easy to see that the directed set $\langle[\mathfrak{c}]^{<\omega},\subseteq\rangle$ is the maximal Tukey type among all directed partial orderings of cardinality $\mathfrak{c}$ and  in \cite{Isbell} Isbell proved that there is $\mathcal{U}$ a Tukey-top ultrafilter (i.e. $\mathcal{U}\equiv_T ([\mathfrak{c}]^{<\omega},\subseteq)$). 

It is known that P-points are not Tukey top (see \cite{TukeyTypes}), that is, if $\mathcal{U}$ is a P-point, then $\langle[\mathfrak{c}]^{<\omega},\subseteq\rangle\nleq_T\mathcal{U}$. On the other hand, by Corollary \ref{ShelahisnotP-point} we have that Shelah ultrafilters cannot be P-points. So, we ask:\\

\begin{Question}
Does $\mathcal{U}_{gen}$ satisfy $\langle[\mathfrak{c}]^{<\omega},\subseteq\rangle\nleq_T\mathcal{U}_{gen}$? If this is the case, does any Shelah ultrafilter satisfy the same?
\end{Question}
\quad

The definition of Shelah ultrafilter attempts to axiomate $\mathcal{U}_{gen}$, however, we do not know if this is sufficient to fully characterize it. The next question is motivated by this and the following two Theorems.

\begin{theorem}[Todorcevic see \cite{Semiselectivecoideals}] [LC]
Every Ramsey ultrafilter is $\mathcal{P}(\omega)/fin$-generic over $L(\mathbb{R})$.
\end{theorem}

\begin{theorem}[Chodounský, Zapletal \cite{Idealsandtheirgenultrafilter}][LC]
Let $\mathcal{I}$ be an $F_\sigma$ ideal on $\omega$ and let $\mathcal{U}$ be an ultrafilter on $\omega$ disjoint from $\mathcal{I}$. The following are equivalent:
\begin{enumerate}[\hspace{0.5cm} (1)]
    \item $\mathcal{U}$ is $\mathbb{P}(\mathcal{I})$-generic over $L(\mathbb{R})$\footnote{For every ideal $\mathcal{I}$ on $\omega$, we denote by  $\mathbb{P}(\mathcal{I})$ the set of all $\mathcal{I}$-positive subsets of $\omega$ ordered by
inclusion.}.
    \item $\mathcal{U}$ is P-point and for every simplicial complex $\mathcal{K}$ on $\omega$, we have that $\mathcal{U}$ contains either a $\mathcal{K}$-set or a set all of whose $\mathcal{K}$-subsets belong to $\mathcal{I}$\footnote{A \emph{simplicial complex} on $\omega$ is a collection of finite subsets of $\omega$ closed
under subset. If $\mathcal{K}$ is a simplicial complex on $\omega$, then a $\mathcal{K}$-set is a set $A\subset\omega$ such
that every finite subset of $A$ is in $\mathcal{K}$.}.
\end{enumerate}
\end{theorem}
\quad

\begin{Question}[LC] Let $\mathcal{U}$ be a Shelah ultrafilter. Is there a $\mathbb{P}$-generic filter over $L(\mathbb{R})$, let say $G$, such that $\mathcal{U}=\bigcup_{p\in G}\mathcal{F}(p)$? If not, is there a combinatorial characterization of the generic filter added by $\mathbb{P}$?
\end{Question}
\quad

Forcing with $\mathbb{P}$ gives us an ultrafilter that avoids isomorphic copies of the ideals $fin^\alpha$ for every $\alpha\in\omega_1$. However, $\mathbb{P}$ does not seem to be the most natural forcing  to add such ultrafilter. We could try to simplify the forcing by using the partial order
$\mathbb{Q}=\{\mathcal{F}:(\mathcal{F}\textit{ is a filter})\wedge(\exists\alpha<\omega_1)(\mathcal{F}^*\textit{ is isomorphic to }fin^\alpha)\}$ ordered by inclusion. So, we could ask if $\mathbb{Q}$ satisfies the same properties of $\mathbb{P}$ as a forcing.
\quad

\begin{question}
Is $\mathbb{Q}$ $\sigma$-closed?
\end{question}
\quad

Even when $\mathbb{Q}$ is more natural, it seems that the Definition of $\mathbb{P}$ makes it behave nicely.

\section{Acknowledgment}
We would like to thanks Michael Hrušák for valuable comments and hours of
stimulating conversations. 

{\normalsize

\printbibliography
}
\end{document}